\documentclass{article}

\usepackage{pascal}
\usepackage{pascalarxiv}

\newcommand{\TheTitle}{Fast and Stable Pascal Matrix Algorithms}

\title{\TheTitle}

\author{Samuel F. Potter\footnotemark[2]\ \footnotemark[3] \and Ramani Duraiswami\footnotemark[3]}

\begin{document}

\renewcommand{\thefootnote}{\fnsymbol{footnote}}
\footnotetext[2]{Corresponding author: sfp@umiacs.umd.edu.}
\footnotetext[3]{Department of Computer Science and Institute for
  Advanced Computer Studies, University of Maryland, College Park,
  MD.}  \renewcommand{\thefootnote}{\arabic{footnote}}

\maketitle

\begin{abstract}
  In this paper, we derive a family of fast and stable algorithms for
  multiplying and inverting $n \times n$ Pascal matrices that run in
  $O(n \log^2 n)$ time and are closely related to De Casteljau's
  algorithm for B\'{e}zier curve evaluation. These algorithms use a
  recursive factorization of the triangular Pascal matrices and
  improve upon the cripplingly unstable $O(n \log n)$ fast Fourier
  transform-based algorithms which involve a Toeplitz matrix
  factorization. We conduct numerical experiments which establish the
  speed and stability of our algorithm, as well as the poor
  performance of the Toeplitz factorization algorithm. As an example,
  we show how our formulation relates to B\'{e}zier curve evaluation.
\end{abstract}


\section{Introduction}

In this paper, we derive fast and stable algorithms for a family of
structured matrices, providing algorithms for matrix-vector
multiplication and inversion which are competitive in terms of their
time requirements and far more stable than previous algorithms. In
particular, we consider the upper and lower triangular Pascal
matrices, introduce $\ell_\infty$-normalized versions of them, and
consider the Bernstein matrices, which generalize the
$\ell_\infty$-normalized Pascal matrices. For $n \times n$ instances
of these matrices (and their transposes, inverses, and
inverse-transposes), we show empircally that a well-known $O(n^2)$
algorithm is extremely fast for small problem sizes, and is
stable. Our main result is the development of an $O(n \log^2 n)$
matrix-vector product for all of these matrices. We then use this
algorithm to obtain an $O(dn \log n)$ algorithm for evaluation of
B\'{e}zier curves of degree $n$ and dimension $d$.

Pascal matrices have seen some attention in recent years due to
applications in numerical computing and signal
processing. Applications include digital filter
design~\cite{psenicka2002bilinear}, coding
theory~\cite{stankovic2016pascal}, computer-aided
design~\cite{bezerra-sacht:computing-bezier-curves,
  bezerra2013efficient, delgado2015accurate}, image
processing~\cite{aburdene2005discrete}, discrete polynomial
transforms~\cite{aburdene1994unification}, and the translation theory
of the fast multipole
method~\cite{tang:fast-transforms-based-on-structured-matrices,
  fmm-book}.

The layout of the paper follows. In the following section, we
introduce our notation, the matrices that we will work with, and
collect some basic results. Following that, we introduce previous
algorithms; in particular, we present a useful $O(n^2)$ algorithm in a
more complete form than has appeared in the literature before, and
review the Toeplitz matrix factorization-based $O(n \log n)$
algorithm. In the fourth section, we present a recursive matrix
factorization which immediately yields our $O(n \log^2 n)$ algorithm;
we also consider some details related to its implementation. The last
three sections of the paper contain our numerical experiments,
followed by applications of our algorithm to the problem of B\'{e}zier
curve interpolation.


\section{Notation}

In this section, we introduce a normalized Pascal matrix which will be
seen to be a special case of Bernstein's matrix. This normalized
Pascal matrix occurs naturally in the translation theory of the fast
multipole method. Additionally, multiplication by Bernstein's matrix
is very closely related to De Casteljau's stable quadratic algorithm
for B\'{e}zier curve interpolation. We will consider this in more
detail later; however, since the essential details of the algorithms
we present in this work are unchanged by restricting our attention to
Pascal's matrix, we do so for the time being and for the sake of
clarity.

Throughout, we assume matrices and vectors to be zero-indexed. By
filling an $n \times n$ matrix with the rows of Pascal's triangle, we
arrive at three different matrices. If we start by defining
$\m{P} = \m{P}_n \in \R^{n \times n}$ by:
\begin{displaymath}
  \m{P}_{ij} = \begin{cases}
    {i \choose j} & \mbox{if } j \leq i, \\
    0 & \mbox{otherwise},
  \end{cases}
\end{displaymath}
referring to this as the $n \times n$ \emph{lower-triangular Pascal
  matrix}, we then have an \emph{upper-triangular Pascal matrix}
$\m{P}^\top$, and a \emph{symmetric Pascal matrix}
$\m{P}\m{P}^\top$. The fact that:
\begin{displaymath}
  {(\m{P}\m{P}^\top)}_{ij} = {i + j \choose j}
\end{displaymath}
is well-known, and a variety of illuminating proofs have been
explored~\cite{edelman-strang:pascal-matrices}.

The structure of Pascal's triangle endows these matrices with many
attractive properties which can be put toward algorithmic ends. A
property of central importance follows. If we define:
\begin{displaymath}
  \m{W} = \m{W}_n = \diag{({(-1)}^k)}_{k=0}^{n-1},
\end{displaymath}
then $\m{P}^{-1} = \m{W} \m{P} \m{W}$. This was proven some time ago
by Call and Velleman~\cite{call-velleman:pascals-matrices}, and casts
the inversion of each type of Pascal matrix in terms of matrix
multiplication. We will see that one family of algorithms neatly
handles both multiplication and inversion.

The simplest way of constructing Pascal's triangle is to simply use
the fact that each element is the sum of its two parents to
recursively extrapolate the triangle. This idea can be profitably
converted into a factorization of $\m{P}$. Letting
$\m{E} = \m{E}_n \in \R^{n \times n}$ be defined by:
\begin{displaymath}
  \m{E}_{ij} = \begin{cases}
    1 & \mbox{if } i = j \mbox{ or } i = j + 1, \\
    0 & \mbox{otherwise},
  \end{cases}
\end{displaymath}
we can see that $\m{E}$ is a bidiagonal matrix with ones on the
diagonal and subdiagonal and that $\m{P}$ satisfies the recurrence:
\begin{displaymath}
  \m{P}_n = \begin{pmatrix}
    1 & \\ & \m{P}_{n-1}
  \end{pmatrix} \m{E}_n.
\end{displaymath}
If we recursively apply this factorization to its maximum extent, we
obtain:
\begin{equation}\label{eq:E-factorization}
  \m{P}_n = \begin{pmatrix}
    \m{I}_{n-2} & \\ & \m{E}_{2}
  \end{pmatrix} \begin{pmatrix}
    \m{I}_{n-3} & \\ & \m{E}_3
  \end{pmatrix} \cdots \begin{pmatrix}
    \m{I}_2 & \\ & \m{E}_{n-2}
  \end{pmatrix} \begin{pmatrix}
    1 & \\ & \m{E}_{n-1}
  \end{pmatrix} \m{E}_n.
\end{equation}
This factorization is well-known and can be used to implement a
space-efficient algorithm for the multiplication of $\m{P}_n$ that
requires only floating point additions and runs in $O(n^2)$
time~\cite{bezerra-sacht:computing-bezier-curves,edelman-strang:pascal-matrices}. Our
goal will be to develop a recursive factorization inspired by
\cref{eq:E-factorization} that will enable an $O(n \log^2 n)$ matrix
multiplication. There are $O(n \log n)$ algorithms and other
$O(n \log^2 n)$ algorithms, but these algorithms are all woefully
unstable and hence unusable for computations large enough to benefit
from their asymptotic time
savings~\cite{tang:fast-algorithms-pascal-matrices,
  wang-zhou:fast-eigenvalue-algorithm,
  wang-linzhang:fast-algorithm-for-solving-linear-systems-of-pascal-type}. Our
factorization will prove to be competitive with $O(n \log n)$
algorithms due to its small asymptotic constant and is stable enough
for practical use.

Although the literature has focused on the Pascal matrix $\m{P}$, we
have found in applications that a normalized version of this matrix is
sometimes more natural, helping to issues with numerical underflow and
overflow typically encountered. As an added benefit, algorithms that
deal with this matrix are somewhat more stable.
\begin{definition}
  The \emph{$\ell_\infty$-normalized lower triangular Pascal matrix}
  $\m{Q} = \m{Q}_n \in \R^{n \times n}$ is given by:
  \begin{displaymath}
    \m{Q}_{ij} = \begin{cases}
      2^{-i} {i \choose j} & \mbox{if } i \geq j, \\
      0 & \mbox{otherwise}.
    \end{cases}
  \end{displaymath}
  As before, we define the corresponding upper triangular and
  symmetric matrices by $\m{Q}^\top$ and $\m{Q}\m{Q}^\top$. We note
  that the entries of $\m{Q}\m{Q}^\top$ are given by:
  \begin{displaymath}
    (\m{Q}\m{Q}^\top)_{ij} = 2^{-i-j} {i + j \choose j}.
  \end{displaymath}
\end{definition}

While the antidiagonals of the symmetric Pascal matrix
$\m{P}\m{P}^\top$ are comprised of the rows of Pascal's triangle, the
antidiagonals of $\m{Q}\m{Q}^\top$ are the rows of Pascal's triangle
normalized so they sum to unity (i.e. $\ell_1$-normalized). Likewise,
where the entries of each of the triangular Pascal matrices can be
construed in terms of the rows of Pascal's triangle, the entries of
the $\ell_\infty$-normalized versions of these matrices can be thought
of as having been constructed from the rows of a normalized Pascal's
triangle.

\begin{note}[inversion of $\m{Q}$]
  The normalized and unnormalized Pascal matrices can be related by a
  diagonal scaling matrix. We use the following notation: for a scalar
  $\delta$, we write:
  \begin{displaymath}
    \m{D}^{(\delta)} = \m{D}^{(\delta)}_n = \diag{(\delta^k)}_{k=0}^{n-1}.
  \end{displaymath}
  Then, $\m{Q} = \m{D}^{(1/2)}\m{P}$ holds. This relation lets us
  apply the preceding identity for matrix inversion:
  \begin{displaymath}
    \m{Q}^{-1} = {(\m{D}^{(1/2)}\m{P})}^{-1} = \m{P}^{-1} {(\m{D}^{(1/2)})}^{-1} =
    \m{W}\m{P}\m{D}^{(2)}\m{W}.
  \end{displaymath}
\end{note}


\section{Prior Art}

Two main sets of algorithms have been considered previously for the
multiplication and inversion of $\m{P}$ and $\m{Q}$. The algorithms
for each of $\m{P}$ and $\m{Q}$ are materially the same, so we will
restrict our attention to $\m{Q}$, since the algorithms for
multiplying and inverting $\m{P}$ are a slight simplification of the
latter.

The first set of algorithms involves factoring $\m{Q}$ into a product
of bidiagonal matrices, and the second factors $\m{Q}$ into a Toeplitz
matrix. Although the former algorithms take $O(n^2)$ time to multiply
or invert $\m{Q}_n$, they require very few additions and
multiplications (the asymptotic constant is quite small) and can be
made to run inplace---that is, using $O(1)$ additional space beyond
the $O(n)$ space required for storage of the input and output
vectors. This set of algorithms has the virtue of being extremely
stable, which has been observed in
practice~\cite{xiao-guang:new-algorithm-linear-systems-pascal}, and
also explored theoretically in some
depth~\cite{alonso2013conditioning}. While this is not explored here,
they also parallelize well.

On the other hand, despite their $O(n \log n)$ time complexity, the
latter algorithms are extremely unstable, so much so that they are
unusable for all but the smallest $n$. In fact, as will be explored
later, the quadratic algorithms discussed above are fast enough for
small problem sizes on modern computer architectures that they are
faster and more stable for every problem size for which the
$O(n \log n)$ algorithms are stable enough to be used, rendering these
FFT-based algorithms useless. Part of this is due to their
implementation and consequently \emph{large} asymptotic constant.

We first present the quadratic algorithms, since our own algorithm
requires them as a base case in a recursion. We also present the
$O(n \log n)$ algorithm and a variation of it for context and for sake
of comparison since we compare our algorithm with it numerically.

\subsection{Quadratic Algorithms for Small
  Problems}\label{sec:quadratic-algs}

We first present a family of quadratic algorithms for the
multiplication and inversion of $\m{P}$ and $\m{Q}$ (limiting our
attention to $\m{Q}$). Although the factorizations are not novel, we
present the algorithms here in a more general and complete way than
has appeared in the literature.

Each algorithm is based on one of the following bidiagonal matrix
factorizations:
\begin{theorem}\label{theorem:E-factorization}
  With $n$ fixed, let $k$ be such that $1 \leq k < n$, and define
  $\m{E}_k^{(\delta)} \in \R^{n \times n}$ and
  $\m{F}_k^{(\delta)} \in \R^{n \times n}$ by:
  \begin{displaymath}
    \m{E}_k^{(\delta)} = \begin{bmatrix}
      \m{I}_{k-1} & \\
      & \begin{pmatrix}
        1 & & & \\
        \delta & \delta & & \\
        & \ddots & \ddots & \\
        & & \delta & \delta
      \end{pmatrix}
    \end{bmatrix}, \qquad \m{F}_k^{(\delta)} = \begin{bmatrix}
      \m{I}_{k-1} & \\
      & \begin{pmatrix}
        1 & & & \\
        1 & \delta & & \\
        & \ddots & \ddots & \\
        & & 1 & \delta
      \end{pmatrix}
    \end{bmatrix}.
  \end{displaymath}
  Then, the following hold:
  \begin{displaymath}
    \m{D}^{(\delta)} \m{P} = \m{E}_{n-1}^{(\delta)} \m{E}_{n-2}^{(\delta)} \cdots \m{E}_2^{(\delta)} \m{E}_1^{(\delta)}, \qquad \m{P} \m{D}^{(\delta)} = \m{F}_{n-1}^{(\delta)} \m{F}_{n-2}^{(\delta)} \cdots \m{F}_2^{(\delta)} \m{F}_1^{(\delta)}.
  \end{displaymath}
\end{theorem}

\begin{proof}
  Both of these factorizations can be proved straightforwardly by
  induction. Proofs can be found in the
  literature~\cite{edelman-strang:pascal-matrices,
    bezerra-sacht:computing-bezier-curves}.
\end{proof}

Our goal is to compute matrix-vector products involving
$\m{Q}, \m{Q}^\top, \m{Q}^{-1}$, and $\m{Q}^{-\top}$, and
\cref{theorem:E-factorization} immediately yields a pair of quadratic
algorithms for each matrix. To see this, we note that the matrices
$\m{E}_k$ and $\m{F}_k$ can be multiplied in $O(k)$ time, and linear
systems involving these matrices can be solved in $O(k)$ time using
forward substitution. Hence, $\m{Q}$ can be multiplied directly in
$O(n^2)$ time by multiplying $n - 1$ bidiagonal matrices. At the same
time, we can rewrite $\m{Q}$ as:
\begin{displaymath}
  \m{Q} = \m{D}^{(1/2)} \m{P} = \squareb{\parens{\m{D}^{(1/2)} \m{P}}^{-1}}^{-1} = \parens{\m{W}\m{P}\m{D}^{(2)}\m{W}}^{-1} = \m{W} {(\m{F}_{n-1}^{(2)})}^{-1} \cdots{} {(\m{F}_1^{(2)})}^{-1} \m{W},
\end{displaymath}
allowing us to solve $n - 1$ systems in $O(n^2)$ time to compute the
same product. We present each type of matrix decomposition for
$\m{Q}, \m{Q}^\top, \m{Q}^{-1}$, and $\m{Q}^{-\top}$ in
\cref{table:Q-factorizations} to illustrate this idea. We note that
different decompositions of the same matrix do \emph{not} necessarily
result in the same algorithm once implemented. The different
algorithms may have difference performance characteristics which are
sensitive to differences in computer architecture (see
\cref{fig:ratios}).

\begin{table}
  \centering
  \begin{tabular}{c|c|c|c|c}
    Method & $\m{Q}$ & $\m{Q}^\top$ & $\m{Q}^{-1}$ & $\m{Q}^{-\top}$ \\
    \midrule
    Multiply & $\m{E}_{n-1} \cdots \m{E}_1$ & $\m{E}_1^\top \cdots \m{E}_{n-1}^\top$ & $\m{W} \m{F}_{n-1} \cdots \m{F}_1 \m{W}$ & $\m{W} \m{F}_1^\top \cdots \m{F}_{n-1}^\top \m{W}$ \\
    Solve & $\m{W} \m{F}^{-1}_1 \cdots \m{F}^{-1}_{n-1} \m{W}$ & $\m{W} \m{F}^{-\top}_{n-1} \cdots \m{F}^{-\top}_1 \m{W}$ & $\m{E}_1^{-1} \cdots \m{E}_{n-1}^{-1}$ & $\m{E}_{n-1}^{-\top} \cdots \m{E}_1^{-\top}$
  \end{tabular}
  \vspace{1em}
  \caption{Matrix factorizations leading to $O(n^2)$ in-place
    algorithms for multiplying $\m{Q}$ and its variations (we take
    $\m{E}_k = \m{E}_k^{(1/2)}$ and
    $\m{F}_k = \m{F}_k^{(2)}$).}\label{table:Q-factorizations}
\end{table}

Note that each algorithm derived from the factorizations in
\cref{table:Q-factorizations} can made to run using $O(1)$ extra space
beyond the $O(n)$ space required by the input and output vectors. In
fact, each algorithm can be made to operate in-place on the input
vector. This is valuable, since in-place algorithms exhibit greater
cache coherency and typically run faster than their out-of-place
equivalents on modern computer architectures.

To fix ideas, we present the in-place multiplication-based algorithm
for computing $\m{x} \gets \m{Q}^{-\top} \m{x}$ below. The algorithm
requires $O(n^2)$ operations overall, using $2 \floor{n/2}$ bit flips
to multiply each matrix $\m{W}$, $n(n-1)/2$ floating point adds, and
$n(n-1)/2$ floating point multiplies total. The seven other algorithms
that can be derived from \cref{table:Q-factorizations} are similar,
and the eight algorithms that deal with $\m{P}$ are simpler still, and
even more efficient.

\begin{algorithm}[H]
  \caption{Compute $\m{Q}^{-\top} \m{x}$ in place using the
    factorization
    $\m{Q}^{-\top} = \m{W} \m{F}_1^\top \cdots \m{F}_{n-1}^\top
    \m{W}$.}\label{algorithm:qit-quad-alg}
  \begin{algorithmic}
    \REQUIRE{} $\m{x} \in \R^n$
    \ENSURE{} $\m{x} \gets \m{Q}^{-\top} \m{x}$
    \STATE{} $i_{\operatorname{max}} \gets 2 \floor{\tfrac{n}{2} - 1} + 1$
    \FOR{$i = 1, 3, \hdots, i_{\operatorname{max}}$}
      \STATE{} $x_i \gets -x_i$
    \ENDFOR{}
    \FOR{$i = n - 2, n - 3, \hdots, 0$}
      \FOR{$j = i, i + 1, \hdots, n - 2$}
        \STATE{} $x_j \gets x_j + x_{j+1}$
        \STATE{} $x_{j+1} \gets 2 x_{j+1}$
      \ENDFOR{}
    \ENDFOR{}
    \FOR{$i = 1, 3, \hdots, i_{\operatorname{max}}$}
      \STATE{} $x_i \gets -x_i$
    \ENDFOR{}
  \end{algorithmic}
\end{algorithm}

\begin{figure}
  \begin{tikzpicture}
    \begin{axis}[
        mark options={solid},
        width=0.8\linewidth,
        height=0.5\linewidth,
        legend style={
          cells={anchor=east},
          legend pos=outer north east
        },
        xlabel=$n$,
        ylabel=$\hat{t}_{\operatorname{direct}}/\hat{t}_{\operatorname{solve}}$]
      \addplot[style=solid, mark=triangle] table[x=n, y=normal] {ratios.dat};
      \addlegendentry{$\m{Q}$}
      \addplot[style=dashed, mark=diamond] table[x=n, y=transpose] {ratios.dat};
      \addlegendentry{$\m{Q}^\top$}
      \addplot[style=densely dotted, mark=square] table[x=n, y=inverse] {ratios.dat};
      \addlegendentry{$\m{Q}^{-1}$}
      \addplot[style=dotted, mark=Mercedes star] table[x=n, y=invtrans] {ratios.dat};
      \addlegendentry{$\m{Q}^{-\top}$}
    \end{axis}
  \end{tikzpicture}
  \caption{Ratios of runtimes for algorithms for multiplying
    $\m{Q}, \m{Q}^\top, \m{Q}^{-1}$, and $\m{Q}^{-\top}$ based on the
    factorizations for problem sizes $n = 16, 32, 48, \hdots, 640$ in
    \cref{table:Q-factorizations}. The graphs show the ratio
    $\hat{t}_{\operatorname{direct}}/\hat{t}_{\operatorname{solve}}$
    of the time taken using the direct method (the first row of
    \cref{table:Q-factorizations}) to the ``solve'' method (the second
    row). Times used are the minimum over several runs. From the plot,
    it is clear that the choice of factorization has an impact on
    performance; the corresponding effect on numerical stability is
    neglible and hence not included.}\label{fig:ratios}
\end{figure}
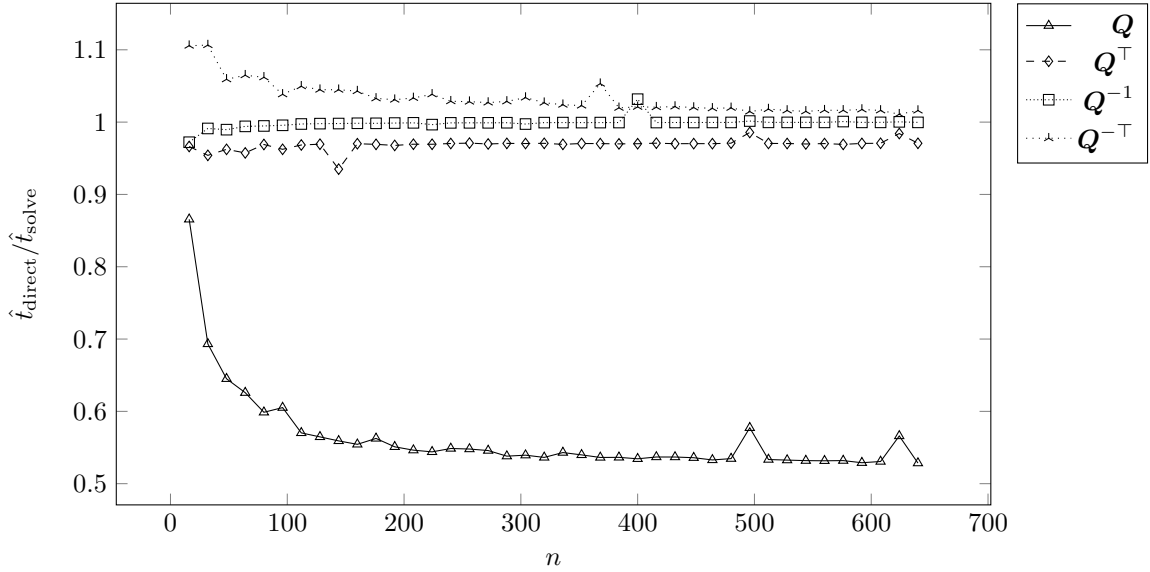

\subsection{Algorithms Based on Toeplitz Factorizations}\label{sec:toep-algs}

It was observed that the matrix $\m{P}$ could be factorized into a
product of diagonal matrices and a Toeplitz
matrix~\cite{wang-linzhang:fast-algorithm-for-solving-linear-systems-of-pascal-type,
  wang-zhou:fast-eigenvalue-algorithm}. Our goal here is to show how
our algorithm relates to De Casteljau's algorithm, suggesting natural
overtures between the two. Since any Toeplitz matrix can be
periodically extended to a circulant matrix, the product of an
$n \times n$ Toeplitz matrix and a vector can be computed in
$O(n \log n)$ time using the fast Fourier transform (FFT). An
immediate problem with this approach is that the product is of size
$2n$. What's more, the algorithm is extremely unstable due to the
particular form the factorization takes. Different attempts have been
made to address the issue of stability, but these attempts have only
yielded slight improvement, and are effectively
unusable~\cite{wang-zhou:fast-eigenvalue-algorithm,
  tang:fast-algorithms-pascal-matrices}. We collect these results here
since we will compare our algorithms with them numerically.

The original
factorization~\cite{wang-linzhang:fast-algorithm-for-solving-linear-systems-of-pascal-type}
results in a matrix-vector product that is so unstable that it is
unusable for any but the smallest problem sizes. A cleverly modified
factorization that is slightly less unstable, but still not usable, is
also presented~\cite{wang-zhou:fast-eigenvalue-algorithm}. We present
this modification here, since it has seen the most use in the
literature. The idea is to introduce a parameter $\alpha > 0$ and
modify the factorization in a way that allows us to minimize the
relative magnitudes of the largest and smallest entries of the
Toeplitz matrix in the factorization. In what follows, the original
factorization can be obtained by taking $\alpha = 1$.

For the vector $\m{a} = {(a_i)}_{i=0}^{n-1}$, we employ the following
notation for lower-triangular Toeplitz matrices:
\begin{displaymath}
  \toeplitz(\m{a}) = \begin{pmatrix}
    a_0 & & & \\
    a_1 & a_0 & & \\
    \vdots & \vdots & \ddots & \\
    a_{n-1} & a_{n-2} & \cdots & a_0
  \end{pmatrix},
\end{displaymath}
and for a fixed scalar $\alpha > 0$, we define:
\begin{displaymath}
  \m{f}_n^{(\alpha)} = {(\alpha^k/k!)}_{k=0}^{n-1} \in \R^n, \qquad \m{\Lambda}_n^{(\alpha)} = \diag{(\m{f}_n^{(\alpha)})}, \qquad \m{T}_n^{(\alpha)} = \toeplitz{(\m{f}_n^{(\alpha)})}.
\end{displaymath}
With these definitions, we have the following factorization of
$\m{P}_n$.

\begin{theorem}\label{theorem:Toeplitz-factorization}
  For $\alpha > 0$, the lower triangular Pascal matrix $\m{P}_n$
  satisfies
  $\m{P}_n = {(\m{\Lambda}_n^{(\alpha)})}^{-1} \m{T}_n^{(\alpha)}
  \m{\Lambda}_n^{(\alpha)}$.
\end{theorem}

\begin{proof}
  A proof of this decomposition can be found elsewhere in the
  literature~\cite{wang-zhou:fast-eigenvalue-algorithm}.
\end{proof}

Next, let $\m{X}_n \in \C^{n \times n}$ denote a discrete Fourier
transform (DFT) matrix (the exact form is not important). Since any
Toeplitz matrix can be periodically extended to a circulant matrix,
the preceding decomposition is equivalent to:
\begin{equation}\label{eq:Toeplitz-factorization}
  \m{P}_n = \begin{pmatrix}
    {(\m{\Lambda}_n^{(\alpha)})}^{-1} \\ \m{0}_{n \times n}
  \end{pmatrix}^\top \m{X}_{2n}^{-1} \diag \parens{\m{X}_{2n} \begin{pmatrix}
      \m{f}_n^{(\alpha)} \\ \m{0}_{n \times 1} \end{pmatrix}}
  \m{X}_{2n} \begin{pmatrix} \m{\Lambda}_n^{(\alpha)} \\ \m{0}_{n
      \times n}
  \end{pmatrix}
\end{equation}
by the convolution theorem (i.e., $\m{X}_n$ diagonalizes circulant
matrices). The parameter $\alpha$ depends on $n$ and is chosen in
order to minimize the relative magnitude of the largest and smallest
entries in the factorization. By setting $\alpha$ equal to:
\begin{displaymath}
  \alpha_{\operatorname{opt}} = \argg \min_{1 \leq \alpha < n - 1} \max\parens{\frac{\alpha^\alpha}{\alpha!}, \frac{\alpha^\alpha (n-1)!}{\alpha^{n-1} \alpha!}},
\end{displaymath}
an approximate minimizer is obtained
~\cite{wang-zhou:fast-eigenvalue-algorithm}. This parameter can then
be computed numerically and cached. The authors also provide a
heuristic for an initial guess for $\alpha$ when computing $\alpha$
numerically:
\begin{displaymath}
  \alpha_{\operatorname{opt}} \approxeq \frac{n - 1}{e}.
\end{displaymath}
Using the fast Fourier transform (FFT), equation
\cref{eq:Toeplitz-factorization} gives a straightforward $O(n \log n)$
algorithm for multiplying $\m{P}$, $\m{Q}$, or any of the other
variants discussed earlier. Unfortunately, as we will see, this
algorithm is unstable.


\section{Recursive Algorithms}

In this section, we introduce a novel family of serial algorithms
which are both stable and efficient. We provide a multiway recursive
factorization of $\m{Q}$ which is the main theorem of this paper and
which enables the development of our algorithms.

\begin{definition}
  Let $\m{b}_0 = 1$, $\m{b}_1 = (1/2, 1/2)$, and inductively define
  $\m{b}_n = \m{b}_{n-1} * \m{b}_1$ for $n > 1$. The vectors
  $\m{b}_n \in \R^{n+1}$ are sometimes referred to as the impulse
  response of the \emph{($\ell_1$-normalized) binomial filters}. These
  are the filters whose nonzero coefficients are just the $n$th
  binomial coefficients, scaled so that they sum to unity: i.e.,
  ${(\m{b}_n)}_k = 2^{-n} {n \choose k}$.
\end{definition}

To establish our matrix factorization, we first define a particular
convolution matrix from the binomial kernels. Fixing $m < n$, we
define the matrix $\m{B}_{n,m} \in \R^{n - m \times n}$ by:
\begin{displaymath}
  \m{B}_{n,m} = \begin{pmatrix}
    {(\m{b}_m)}_{0} & {(\m{b}_m)}_{1} & \cdots{} & {(\m{b}_m)}_{m} & & \\
    & \ddots & \ddots & & \ddots & \\
    & & {(\m{b}_m)}_{0} & {(\m{b}_m)}_{1} & \cdots{} & {(\m{b}_m)}_{m}
  \end{pmatrix}.
\end{displaymath}
This matrix and its transpose can be conceived of in terms of simple
convolution operations. For example, for a vector $\m{x} \in \R^n$,
the product $\m{B}_{n,m}\m{x}$ is equivalent to the function call
\texttt{conv(x, b, 'valid')} in the MATLAB programming language. Along
the same lines, if $\m{x} \in \R^m$, then $\m{B}_{n,m}^\top \m{x}$ can
be computed by \texttt{conv(b, x)}.

\begin{lemma}\label{lemma:L-factorization-theorem}
  With $n \in \N$ fixed, let $k \in \N$ and let
  $n = n_1 + \cdots + n_k$ be an integer partition of $n$. Define
  $c_0 = 0$ and $c_i = n_1 + \cdots + n_{i}$ for $i > 0$. Let
  $\m{L}_i \in \R^{n \times n}$ be defined by:
  \begin{displaymath}
    \m{L}_i = \begin{bmatrix}
      \m{I}_{c_{i-1}} & \m{0}_{c_{i-1} \times n - c_{i-1}} \\
      \m{0}_{n - c_{i-1} \times c_{i-1}} & \begin{pmatrix}
        \m{Q}_{n_i} & \m{0}_{n_i \times n - c_{i}}  \\
        \midrule
        \multicolumn{2}{c}{\m{B}_{n - c_{i-1}, n_i}}
      \end{pmatrix}
    \end{bmatrix}.
  \end{displaymath}
  Then,
  $\m{Q}_n = \m{L}_k \m{L}_{k-1} \cdots \m{L}_2 \m{L}_1$.
\end{lemma}

\begin{proof}
  The proof in the case of $k = 2$ is available in the literature and
  is straightforward~\cite{bezerra-sacht:computing-bezier-curves}. The
  result follows from recursively applying the $k = 2$ case to the
  general partition $n = n_1 + \cdots + n_k$.
\end{proof}

Although we have presented \cref{lemma:L-factorization-theorem} in
terms of the $\ell_\infty$-normalized Pascal matrix $\m{Q}_n$, it
applies to the unnormalized Pascal matrix $\m{P}_n$ as well. The
corresponding theorem is obtained by defining the matrix $\m{B}_{n,m}$
in terms of the unnormalized binomial filter kernels $2^n \m{b}_n$ and
defining $\m{L}_1, \hdots, \m{L}_k$ using $\m{P}$ instead of
$\m{Q}$. Also, we note that if we take $k = n - 1$ and $n_i = 1$ for
$i = 1, \hdots, k$, we recover \cref{theorem:E-factorization} as a
special case of \cref{lemma:L-factorization-theorem}.

\begin{example}\label{example:Q12}
  Since \cref{lemma:L-factorization-theorem} is a bit opaque, consider
  the case where $k = 4$ and $(n_1, n_2, n_3, n_4) = (3, 2, 3,
  2)$. Then $(c_0, c_1, c_2, c_3, c_4) = (0, 3, 5, 8, 10)$, and the
  decomposition $\m{Q}_{10} = \m{L}_4\m{L}_3\m{L}_2\m{L}_1$ can be
  schematically depicted as follows:
  \begin{center}
    \vspace{1em}
    \includegraphics{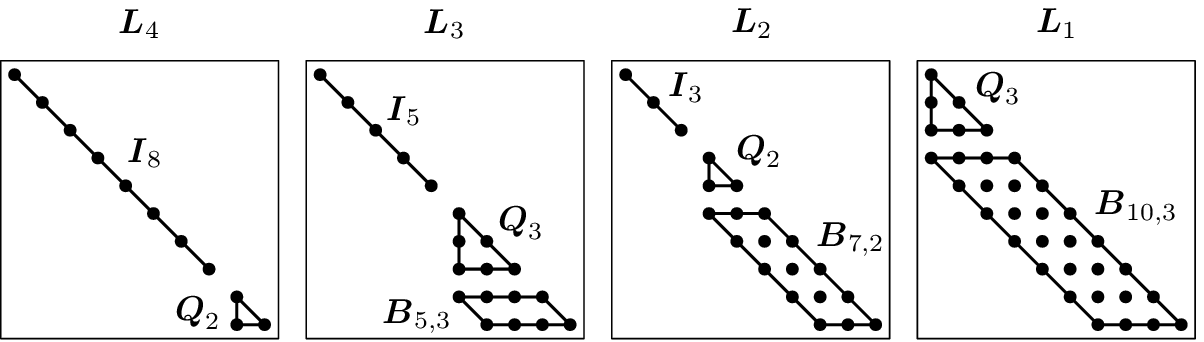}
  \end{center}
\end{example}

If we carry out the block multiplications in
$\m{L}_k \cdots \m{L}_1 = \m{Q}_n$, we obtain a block recursion for
the product $\m{Q}_n\m{x}$.

\begin{theorem}[Main theorem]\label{theorem:main-theorem}
  Let $\m{x} \in \R^n$, let $k \in \N$, and let
  $n = n_1 + \cdots + n_k$ be a partition of $n$. Define
  $c_0, c_1, \hdots, c_k$ as before. Then, the matrix-vector product
  $\m{Q}_n \m{x}$ satisfies:
  \begin{displaymath}
    \m{Q}_n \m{x} = \begin{pmatrix}
      \m{Q}_{n_1} \m{x}_{1:c_1} \\
      \m{Q}_{n_2} \m{B}_{c_2,c_1} \m{x}_{1:c_2} \\
      \vdots \\
      \m{Q}_{n_{k-1}} \m{B}_{c_{k-1},c_{k-2}} \m{x}_{1:c_{k-1}} \\
      \m{Q}_{n_{k}} \m{B}_{c_{k},c_{k-1}} \m{x}_{1:c_{k}}
    \end{pmatrix},
  \end{displaymath}
  where $\m{x}_{s:t} = (x_s, x_{s+1}, \hdots, x_{t-1}, x_t)$.
\end{theorem}

\begin{proof}
  For $k = 1$, there is nothing to prove: the block matrix reduces to
  $\m{Q}_n \m{x}$. Hence, the theorem holds trivially for $k = 1$ and
  all $n \in \N$. Let $k > 1$, $n = n_1 + \cdots + n_k \in \N$, and
  assume the theorem holds. Let $n'_1 \in \N$ and define
  $n' = n'_1 + n = c'_{k+1}$. Using \cref{lemma:L-factorization-theorem}, for
  $\m{x} \in \R^{n'}$ we have:
  \begin{displaymath}
    \m{Q}_{n'} \m{x} = \begin{pmatrix}
      \m{I}_{n'_1} & \m{0}_{n'_1 \times n} \\ \m{0}_{n \times n'_1} & \m{Q}_n
    \end{pmatrix} \begin{pmatrix}
      \m{Q}_{n'_1} & \m{0}_{n'_1 \times n} \\
      \midrule \multicolumn{2}{c}{\m{B}_{{n'_1} + n, n'_1}}
    \end{pmatrix} \m{x} = \begin{pmatrix}
      \m{Q}_{n'_1} \m{x}_{1:n'_1} \\ \m{Q}_n \m{B}_{{n'_1} + n, n'_1} \m{x}_{1:c'_{k+1}}
    \end{pmatrix}
  \end{displaymath}
  If we expand $\m{Q}_n$ in the second block using the inductive
  hypothesis, multiply the convolution matrices, and relabel
  $n'_2 = n_1, n'_3 = n_2, \hdots$, and likewise for $c_i'$, we have
  the result.
\end{proof}

Since it can be identified with convolution, multiplication by
$\m{B}_{n,m}$ can be done using the fast Fourier transform in
$O(n \log n)$ time. For instance, letting $\m{X}_n$ be an $n \times n$
DFT matrix as in \cref{sec:toep-algs} and letting $\m{C}_{n,m}$
periodically extend $\m{B}_{n,m}$ to an $n \times n$ circulant matrix,
we have:
\begin{displaymath}
  \m{B}_{n,m} = \begin{pmatrix} \m{I}_{n-m} & \m{0}_{n - m \times m} \end{pmatrix} \m{C}_{n,m} = \begin{pmatrix} \m{I}_{n-m} & \m{0}_{n - m \times m} \end{pmatrix} \m{X}_n^* \diag \parens{\parens{\frac{1 + e^{-2\pi{}ij/n}}{2}}^m}_{j=0}^{n-1} \m{X}_n,
\end{displaymath}
where we have used
$\m{b}_m = \m{b}_1^{*m} = \m{b}_1 * \cdots * \m{b}_1$ to compute the
DFT vector coefficients $[(1 + e^{-2\pi{}ij/n})/2]^m$.

Since we can compute $\m{B}_{n,m} \m{x}$ in $O(n \log n)$ time, we can
now state an algorithm for the computation of $\m{Q}_n\m{x}$ which is
efficient and numerically stable. Fixing $N \in \N$, and taking
$k = 2$, $n_1 = \floor{n/2}$, and $n_2 = n - \floor{n/2}$ in
\cref{theorem:main-theorem}, repeated recursive application of
\cref{theorem:main-theorem} yields the following algorithm. We will
say more about the choice of $N$ in a later section.

\begin{algorithm}[H]
  \caption{Compute $\m{Q}_n\m{x}$ serially using a two-way
    recursion}\label{algorithm:serial-two-way-recursion}
  \begin{algorithmic}
    \REQUIRE{} $\m{x} \in \R^n$
    \ENSURE{} $\m{x} \gets \m{Q}_n\m{x}$
    \IF{$n \leq N$}
      \STATE{} Compute $\m{x} \gets \m{Q}_n\m{x}$ using the $O(n^2)$
      algorithm derived from \cref{table:Q-factorizations}
    \ELSE{}
      \STATE{} $m \gets \floor{n/2}$
      \STATE{} Compute $\m{x}_{m+1:n} \gets \m{B}_{n,m} \m{x}$ using the FFT
      \IF{$n - m > 1$}
        \STATE{} Recursively compute $\m{x}_{m+1:n} \gets \m{Q}_{n-m} \m{x}_{m+1:n}$
      \ENDIF{}
      \IF{$m > 1$}
        \STATE{} Recursively compute $\m{x}_{1:m} \gets \m{Q}_{m} \m{x}_{1:m}$
      \ENDIF{}
    \ENDIF{}
  \end{algorithmic}
\end{algorithm}

\begin{theorem}
  The matrix-vector product $\m{Q}_n\m{x}$ can be computed in
  $O(n \log^2 n)$ time in exact arithmetic.
\end{theorem}

\begin{proof}
  The product $\m{B}_{n,m}\m{x}$ is a convolution and can be computed
  in $O(n \log n)$ time. Let $B_n$ denote the time required to compute
  $\m{B}_{n,m}\m{x}$, so that $B_n = O(n \log n)$. If we let $T_n$
  denote the time required to compute $\m{Q}_n\m{x}$, then
  $T_n = B_n + T_{n-\floor{n/2}} + T_{\floor{n/2}}$. Solving this
  recurrence using the master method gives $T_n = \Theta(n \log^2 n)$,
  hence $T_n = O(n \log^2 n)$~\cite{clrs}.
\end{proof}

We note that if we taken $k > 2$, and if $k = O_n(1)$, then we still
obtain an algorithm which is $O(n \log^2 n)$. In particular, by using
the Akra-Bazzi
theorem~\cite{akra-bazzi:solution-of-linear-recurrence-equations}, we
show that an approximately uniform integer partition of $n$ yields an
algorithm which runs in $\Theta(k n \log^2 n)$ time. In a later
section, we discuss numerical experiments which suggest that this can
be improved upon.

\begin{theorem}
  Consider the modification of
  \cref{algorithm:serial-two-way-recursion} using a $k$-way partition
  of $n$, where $k > 2$. The modified algorithm runs in
  $\Theta(k n \log^2 n)$ time.
\end{theorem}

\begin{proof}
  Let $n_1 + \cdots + n_k = n$ such that $n_i \in \N$ and either
  $n_i = \floor{n/k}$ or $n_i = \ceil{n/k}$. Let $T_n$ denote the
  runtime of the modified algorithm applied to an input of size
  $n$. We apply the Akra-Bazzi
  theorem~\cite{akra-bazzi:solution-of-linear-recurrence-equations} to
  show that $T_n = \Theta(kn \log^2 n)$. We have that $T_n$ satisfies:
  \begin{displaymath}
    T_n = \sum_{i=2}^k B_{c_i} + \sum_{i=1}^k T_{n_i} = \sum_{i=2}^k B_{in/k} + \sum_{i=1}^k T_{n/k},
  \end{displaymath}
  where we ignore the floors and ceilings (the Akra-Bazzi theorem is
  designed to account for this). Since $B_n = \Theta(n \log n)$:
  \begin{displaymath}
    \sum_{i=2}^k B_{in/k} = \sum_{i=2}^k \Theta\parens{\frac{in}{k} \log \frac{in}{k}} = \frac{1}{k} {k+1 \choose 2} \Theta(n \log n) = \Theta(k n \log n).
  \end{displaymath}
  Applying the Akra-Bazzi theorem involves determing a constant $p$
  from the recursion and then evaluating an integral which gives the
  asymptotic estimate for $T_n$. In our case, $p = 1$, and the
  integral to compute is:
  \begin{displaymath}
    \int_1^n \frac{k u \log u}{u^{p+1}} du = k \int_1^n \frac{\log u}{u} du = \Theta(k \log^2 n).
  \end{displaymath}
  Applying the theorem yields $T_n = \Theta(k n \log^2 n)$.
\end{proof}

There is some question as to whether this estimate for $T_n$ is
pessimistic. Computing the optimal partition at each step can be done
using dynamic programming, but this is far from ideal. In the next
section, we discuss our numerical attempts at doing so. Our results
seem to suggest that for $k > 2$, the uniform partition of $n$ is
suboptimal, and that the optimal partition in fact results in an
algorithm which may perform as well as the two-way recursion using the
uniform partition.

\begin{conjecture}\label{conj:optimal-split}
  Let $n \gg k > 2$ and let
  $\lambda_1, \hdots, \lambda_k \in \mathbb{Q}$ such that
  $\lambda_i > 0$ for each $i$ and
  $\lambda_1 + \cdots + \lambda_k = 1$. Let $n_i = \lambda_{i}n$ for
  each $i$---i.e., let $\lambda_1, \hdots, \lambda_k$ determine a
  partition of $n$. Then, in the above proof, there exists a way of
  choosing $\lambda_1, \hdots, \lambda_k$ so that
  $T_n = \Theta(k n \log^2 n)$.
\end{conjecture}


\subsection{Optimal Recursion}

The $O(n^2)$ algorithm presented in \cref{sec:quadratic-algs} is
extremely fast for small problem sizes. In designing our
$O(n \log^2 n)$ algorithm, this should be borne in mind---if we carry
out the recursion to its maximum depth, we would compute a large
number of small FFTs when using the small, inplace algorithm would be
significantly faster. In this section, we will consider how to find
the optimal threshold $N$ for the case of the uniform partition when
$k = 2$.

We define the following runtime costs:
\begin{subequations}
  \begin{align*}
    T_n &= \mbox{the cost to compute $\m{Q}_{n}\m{x}$ using the $O(n \log^2 n)$ algorithm}, \\
    A_n &= \mbox{the cost to compute $\m{Q}_{n}\m{x}$ using the $O(n^2)$ algorithm}, \\
    B_n &= \mbox{the cost to compute $\m{B}_{n,m}\m{x}$ using the $O(n \log n)$ algorithm},
  \end{align*}
\end{subequations}
where we found the parameter $m$ in computing $\m{B}_{n,m}\m{x}$ to have
little effect on the runtime. With this setup, we can formulate the
optimal cost of our algorithm as:
\begin{equation}\label{eq:dynprog1}
  T_n = \min \parens{A_n, \min{(A_{\floor{n/2}}, T_{\floor{n/2}})} + \min{(A_{n - \floor{n/2}}, T_{n - \floor{n/2}})} + B_n}.
\end{equation}
This cost can be computed in a straightforward manner using dynamic
programming~\cite{clrs}, and the corresponding recursion can be
recovered and used to generate code. Alternatively, we can determine
this tuning parameter empirically, taking advantage of the performance
gains due to autotuning FFT libraries~\cite{fftw}.

Upon solving \cref{eq:dynprog1}, it becomes clear that the rule which
governs the recursion is simply to recurse if $n$ is above a certain
threshold $N$, and to use the $O(n^2)$ algorithm otherwise. On the
computer we ran our numerical tests on, we found $N = 452$. That is,
if $n \leq N$, then the $O(n^2)$ algorithm is faster; and if $n > N$,
then the recursion is faster. This consideration reflects the
recursive base case in \cref{algorithm:serial-two-way-recursion}.

To compute $N$, we simply solve the following minimization problem:
\begin{displaymath}
  N = \min \set{n \geq 1 : A_n > T_n},
\end{displaymath}
where $T_n$ is computed from \cref{eq:dynprog1}. This can be done
using bottom-up dynamic programming extremely easily. Also, we note
that to approximate the costs $A_n$ and $B_n$, we should have
$A_n = a_0 + a_1 n + a_2 n^2$ and $B_n = b_0 + b_1 n + b_2 n \log n$,
for some constants $a_i$ and $b_i$, with $i = 0, 1, 2$. These
constants can be computed empirically by estimating the CPU time of
each algorithm and solving the corresponding regression problem. This
is the approach we used to estimate $N$ in our tests.

\begin{note}
  There are two variations on \cref{eq:dynprog1}. First, observe
  that the partition of $n$ in the above was fixed as
  $n = \floor{n/2} + (n - \floor{n/2})$. This restriction can be
  removed by considering the following dynamic programming problem:
  \begin{equation}\label{eq:dynprog2}
    T_n = \min_{m \geq 1} \parens{A_n, \min{(A_{m}, T_{m})} + \min{(A_{n - m}, T_{n - m})} + B_n}.
  \end{equation}
  More generally, we could consider partitions of $n$ with more than
  two terms. These dynamic programs can be solved, but take much more
  time to do so, especially the latter. We observed two things when
  conducting these experiments:
  \begin{itemize}
  \item When solving \cref{eq:dynprog2}, we found that the
    recovered recursions were very similar to those recovered when
    solving \cref{eq:dynprog1}. Namely, for $n \gg N$, we recover
    $m = \floor{n/2}$; however, for $n$ just slightly greater than
    $N$, a ``soft transition'' is introduced with $m \neq \floor{n/2}$
    (but with $|m - \floor{n/2}|$ not too large). This makes intuitive
    sense, but incorporating this difference had no noticeable effect
    on CPU time.
  \item For partitions with more than two terms, the optimal split is
    \emph{not} approximately uniform. What's more, the cost of a
    partition with more terms appears to be at least as great as the
    cost of a partition with fewer terms; interestingly (see
    \cref{conj:optimal-split}), at least for the first few $k$, the
    cost obtained using the optimal partition is essentially the same
    for each $k$ when simulating results using cost functions that
    were estimated using regression.
  \end{itemize}
\end{note}


\section{Numerical Experiments}\label{sec:numerical-experiments}

We conducted numerical experiments to assess the performance of our
algorithms and compare them with preexisting methods---namely, the
algorithms presented in \cref{sec:quadratic-algs} and
\cref{sec:toep-algs}. Our numerical stability results and timing
comparisons are presented in \cref{fig:timings-and-stability}. All of
our comparisons were performed on a computer with two dual-core 2.2
GHz Intel Core i7 CPUs with 256KB of L2 cache per core and 4MB of L3
cache. Our algorithms were implemented in C and C++. We used FFTW3 to
compute the fast Fourier transform and the GNU Multiprecision Library
to compute the ground truth results for our stability
tests~\cite{fftw, libgmp}.

First, our tests suggest that all of the Toeplitz algorithms are
essentially useless. The relative error increases extremely rapidly
and they quickly become too unstable to be of any use. In the range
that they provide reasonable results, the quadratic algorithms are
1.5--2 orders of magnitude faster. Given that the quadratic algorithms
only involve simple floating point operations and that the Toeplitz
algorithms require extra setup and memory allocation to call the FFT
routines---\emph{and} since the asymptotic constant for the estimate
of the number of FLOPS is significantly higher than for the quadratic
algorithms---this is unsurprising.

We also tested the performance of using a recursion with $k > 2$. In
particular, we modified \cref{algorithm:serial-two-way-recursion} by
replacing the first step of the recursion with a recursion based on a
uniform partition of $n$ with $k > 2$. We measured the ratio of the
time taken by \cref{algorithm:serial-two-way-recursion} to this
modified algorithm for different $k$ and $n$:
\begin{center}
  \begin{tabular}[h]{ccccc}
    & $n = 2^{13}$ & $2^{14}$ & $2^{15}$ & $2^{16}$ \\
    \hline
    $k = 4$ & 0.9644\ldots & 0.9497\ldots & 1.001\ldots & 0.9635\ldots \\
    6 & 0.8312\ldots & 0.7925\ldots & 0.8078\ldots & 0.8243\ldots \\
    8 & 0.8719\ldots & 0.8535\ldots & 0.9009\ldots & 0.8691\ldots
  \end{tabular}
\end{center}
There are two things to observe:
\begin{itemize}
\item The performance of the resulting algorithms is sensitive to the
  problem size and the partition size---powers of two are better. This
  is because, despite the fact that FFTW uses an $O(n \log n)$
  algorithm for all problem sizes, problem sizes which are highly
  composite are faster to solve.
\item A uniform partition of $n$ does not appear to be optimal for
  $k > 2$.
\end{itemize}
This suggests that by carefully exploring the space of possible
recursions for $k \geq 2$, it may be possible to generate recursions
which are faster than
\cref{algorithm:serial-two-way-recursion}. Another possibility is use
in the design of parallel algorithms: parallel algorithms tend to
perform best when there is little communication among worker threads
and when the granularity of jobs is large. By taking $k$ to be the
number of processors for the first recursion and $k = 2$ afterwards, a
suitable design for a parallel algorithm emerges.

\begin{figure}
  \begin{tikzpicture}
    \begin{groupplot}[
        group style={
          group size=1 by 2,
          x descriptions at=edge bottom,
          y descriptions at=edge left,
          vertical sep=0
        },
        xmin=6e-1,
        xmax=2e5,
        xmode=log,
        ymode=log,
        width=\linewidth,
        xlabel=$n$]
      \nextgroupplot[
          ylabel=\textbf{$||\m{y} - \hat{\m{y}}||_\infty/||\m{y}||_\infty$},
          mark options={solid},
          height=0.4\linewidth,
          legend style={legend pos=north west}]
        \addplot[style=solid, mark=triangle] table[x=N, y=Q_mult_rec] {errors.dat};
        \addlegendentry{Recursive};
        \addplot[style=dashed, mark=diamond] table[x=N, y=Q_mult_small] {errors.dat};
        \addlegendentry{Quadratic};
        \addplot[style=densely dotted, mark=square] table[x=N, y=Q_mult_toep] {errors.dat};
        \addlegendentry{Toeplitz};
      \nextgroupplot[
          ylabel=Time (s.),
          height=0.4\linewidth,
          mark options={solid}]
        \addplot[style=solid, mark=triangle] table[x=N, y=Q_mult_rec] {timings.dat};
        \addplot[style=dashed, mark=diamond] table[x=N, y=Q_mult_small] {timings.dat};
        \addplot[style=densely dotted, mark=square] table[x=N, y=Q_mult_toep] {timings.dat};
    \end{groupplot}
  \end{tikzpicture}
  \caption{Top: uniform relative errors of different methods of
    computing $\m{y} = \m{Q}_n\m{x}$, where
    $\m{x}_i \sim \mathcal{N}(0, 1)$, and for
    $n = 2^0, 2^1, 2^2, \hdots$. Each relative error is the sample
    mean over 10 trials. The ground truth is denoted $\m{y}$, and the
    approximate value is denoted $\hat{\m{y}}$. The ground truth value
    $\m{y}$ was computed using the GMP library using 50 digits of
    precision~\cite{libgmp}. Bottom: timings for different methods of
    multiplying $\m{Q}_n$. Times are the minimum CPU time over 10
    trials. Examining the plots, we can clearly see that the Toeplitz
    algorithm loses all digits of precision before $n = 10^2$; at the
    same time, the quadratic algorithm is faster than the Toeplitz
    algorithm for $n \leq 300$, approximately. We can also see that as
    $n$ approaches $n = 10^5$,
    \cref{algorithm:serial-two-way-recursion} is two orders of
    magnitude faster than the quadratic algorithm, and only one order
    of magnitude slower than the Toeplitz algorithm. For small problem
    sizes, little speed is lost incorprating the quadratic algorithm
    as a recursive base case. A jump in the instability of
    \cref{algorithm:serial-two-way-recursion} is also evident: this
    occurs at $n = N$ and is due to the introduction of the greater
    (but minimal) numerical instability of the
    FFT.}\label{fig:timings-and-stability}
\end{figure}
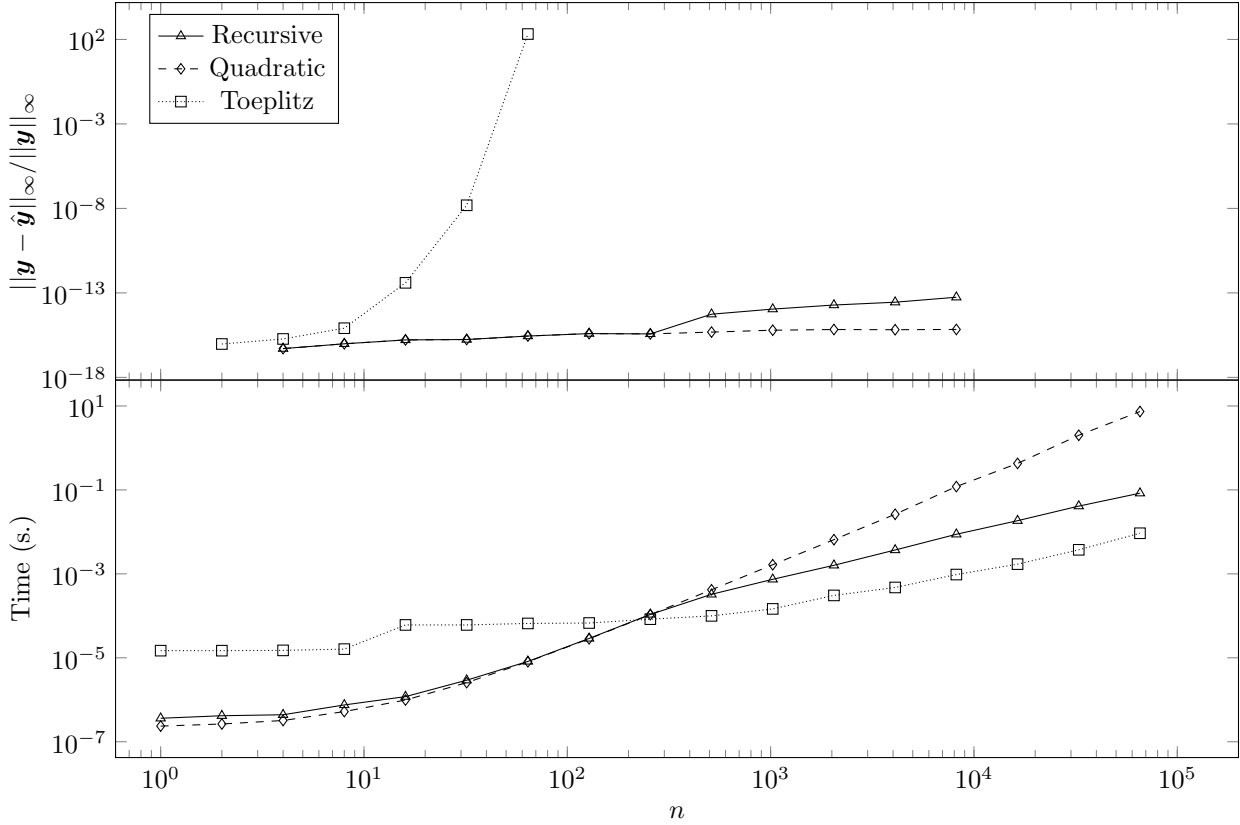


\section{B\'{e}zier Curve Interpolation}

In this section, we consider the problem of B\'{e}zier curve
evaluation, which is readily cast in terms of Pascal matrix
multiplication. B\'{e}zier curves are parametric curves frequently
used in computer graphics and computed-aided modeling which were
invented by Pierre B\'{e}zier~\cite{bezier:numerical-control}. De
Casteljau's algorithm evaluates a B\'{e}zier curve at one point along
the curve in $O(n^2)$ time, where $n$ is the number of control
points. It is also possible to perform this evaluation in $O(n)$ time
using, e.g., Horner's rule and the monomial form of the same
polynomial. Unfortunately, it is well-documented that this approach is
numerically unstable~\cite{farin2002curves}. The algorithm we present
here has been referred to elsewhere as the Bernstein-Fourier algorithm
for B\'{e}zier curve evaluation~\cite{bezerra2013efficient,
  delgado2015accurate} and effectively improves the runtime of De
Casteljau's algorithm to $O(n \log n)$, allowing fast and stable
evaluation of B\'{e}zier curves.

\begin{definition}
  Let $n \in \N$ and fix $t \in \R$ such that $0 \leq t \leq 1$. Then,
  for $i = 0, \hdots, n$, the \emph{$i$th Bernstein polynomial of
    degree $n$} is given by:
  \begin{displaymath}
    b^{(n)}_i(t) = {n \choose i} t^i (1 - t)^{n - i},
  \end{displaymath}
  and the \emph{$n$th Bernstein matrix}
  $\m{\mathcal{B}}^{(t)}_n \in \R^{n+1 \times n+1}$ is the lower
  triangular matrix whose entries are given by:
  \begin{displaymath}
    \parens{\m{\mathcal{B}}^{(t)}_n}_{ij} = b_j^{(i)}(t).
  \end{displaymath}
\end{definition}

Let $d \geq 1$ and $\m{p} \in \R^{n+1 \times d}$ be the matrix
containing $n+1$ control points defining a B\'{e}zier curve
$\gamma : [0, 1] \to \R^d$ as its rows. De Casteljau's algorithm for
B\'{e}zier curve evaluation computes $\gamma(t)$ in $O(dn^2)$
time~\cite{farin2002curves}. The algorithm uses a decomposition which
is materially the same as the decomposition of $\m{Q}_n$ given by
\cref{theorem:E-factorization}. In particular, it computes
$\gamma(t) = \m{e}_{n+1}^\top \m{\mathcal{B}}^{(t)}_n \m{p}$. Hence,
by modifying the algorithms for the multiplication of $\m{Q}_n$
presented above, from this $O(dn^2)$ and $O(dn \log^2 n)$ algorithms
for B\'{e}zier curve evaluation are initially obtained.

We can readily see that for $\delta = 1/2$, we have
$\m{\mathcal{B}}^{(1/2)}_n = \m{D}^{(1/2)}_n \m{P}_n$. Along the lines
of the derivation in \cref{sec:quadratic-algs}, we can define a matrix
similar to $\m{E}_k^{(\delta)}$ which yields a decomposition of
$\m{\mathcal{B}}^{(t)}_n$ into a product of bidiagonal
matrices. Define:
\begin{displaymath}
  \m{G}_k^{(t)} = \begin{bmatrix}
      \m{I}_{k-1} & \\
      & \begin{pmatrix}
        1 & & & \\
        1 - t & t & & \\
        & \ddots & \ddots & \\
        & & 1 - t & t
      \end{pmatrix}
    \end{bmatrix}.
\end{displaymath}
Then
$\m{\mathcal{B}}^{(t)}_n = \m{G}_{n-1}^{(t)} \m{G}_{n-2}^{(t)} \cdots
\m{G}_1^{(t)}$, noting that $\m{G}_k^{(1/2)} = \m{E}_k^{(1/2)}$. This
decomposition yields De Casteljau's algorithm in the same way that the
decompositions in \cref{table:Q-factorizations} led to
\cref{algorithm:qit-quad-alg}.

\begin{algorithm}[H]
  \caption{Evaluate $\gamma(t)$ using De Casteljau's
    algorithm.}\label{algorithm:de-casteljau}
  \begin{algorithmic}
    \REQUIRE{} $t \in \R$ such that $0 \leq t \leq 1$, $\m{p} \in \R^{n + 1 \times d}$
    \ENSURE{} $\m{p}_{n+1,:} \gets \gamma(t)$
    \FOR{$i = 0, \hdots, n - 1$}
      \FOR{$j = n - i + 1, \hdots, n$}
        \STATE{} Set $\m{p}_{j,:} \gets t \cdot \m{p}_{j,:}$
        \STATE{} Set $\m{p}_{j,:} \gets (1 - t) \cdot \m{p}_{j - 1, :}$
      \ENDFOR{}
    \ENDFOR{}
  \end{algorithmic}
\end{algorithm}

Proceeding as before, we can define a matrix analogous to
$\m{B}_{n,m}$ from the matrix $\m{G}_k^{(t)}$. Letting
$\m{c}_0^{(t)} = 0$ and $\m{c}_1^{(t)} = (1 - t, t)$, we inductively
define $\m{c}_i^{(t)}$ by
$\m{c}_i^{(t)} = \m{c}_{i-1}^{(t)} * \m{c}_1^{(t)}$ and the matrix
$\m{C}_{n,m}^{(t)} \in \R^{n - m \times n}$ by:
\begin{displaymath}
  \m{C}_{n,m}^{(t)} = \begin{pmatrix}
    {(\m{c}_m^{(t)})}_{0} & {(\m{c}_m^{(t)})}_{1} & \cdots{} & {(\m{c}_m^{(t)})}_{m} & & \\
    & \ddots & \ddots & & \ddots & \\
    & & {(\m{c}_m^{(t)})}_{0} & {(\m{c}_m^{(t)})}_{1} & \cdots{} & {(\m{c}_m^{(t)})}_{m}
  \end{pmatrix}.
\end{displaymath}
With these matrices in hand, we observe that no essential
characteristics of \cref{lemma:L-factorization-theorem} or
\cref{theorem:main-theorem} are changed by replacing $\m{Q}_n$ with
$\m{\mathcal{B}}^{(t)}_n$ and $\m{B}_{n,m}$ with
$\m{C}_{n,m}^{(t)}$. Correspondingly, replacing these matrices in
\cref{algorithm:serial-two-way-recursion} yields algorithms which
relate to De Casteljau's algorithm for B\'{e}zier curve evaluation in
the same way that \cref{algorithm:serial-two-way-recursion} relates to
the algorithms of \cref{table:Q-factorizations}.

However, we note that the product
$\m{\mathcal{B}}^{(t)}_n \m{p} \in \R^{n + 1 \times d}$ contains
redundant entries which specify other B\'{e}zier curves. The point
$\gamma(t)$ corresponds to the last row of
$\m{\mathcal{B}}^{(t)}_n \m{p}$; we can focus on computing this row
directly to save time. To this end, we provide the following efficient
algorithm.

\begin{algorithm}[H]
  \caption{Evaluate the B\'{e}zier curve $\gamma(t)$ defined by the
    control points $\m{p}^{n + 1 \times d}$.}\label{alg:bez-alg-1}
  \begin{algorithmic}
    \REQUIRE{} $t \in \R$ such that $0 \leq t \leq 1$, $\m{p} \in \R^{n + 1 \times d}$, $k \in \N$ such that $k = O(\sqrt{n \log n})$
    \ENSURE{} $\m{p}_{n+1,:} \gets \gamma(t)$
    \STATE{} Set $\m{p}_{n-k:n+1} \gets \m{C}_{n+1,n+1-k}^{(t)} \m{p}$
    \STATE{} Set $\m{p}_{n-k:n+1} \gets \m{\mathcal{B}}^{(t)}_k \m{p}_{n-k:n+1}$
  \end{algorithmic}
\end{algorithm}

We can expect this algorithm to be stable, owing to the stability
tests performed in \cref{sec:numerical-experiments}. It also runs in
$O(n \log n)$ time, as the next theorem shows. This improves upon the
classical $O(n^2)$ runtime of De Casteljau's algorithm without
compromising its stability beyond the slight numerical instability
introduced by computing the fast Fourier transform.

\begin{theorem}\label{lemma:fast-Bernstein-product-entries}
  Let $n, d \in \N$ such that $d = O_n(1)$ and let
  $\m{p} \in \R^{n + 1 \times d}$. Then, for each $t$ such that
  $0 \leq t \leq 1$ and $i = 0, \hdots, n$, the vector
  $\gamma(t) = \m{e}_{n+1}^\top \m{\mathcal{B}}_n^{(t)} \m{p} \in
  \mathbb{R}^d$ can be computed in $O(d n \log n)$ time.
\end{theorem}

\begin{proof}
  To show that \cref{alg:bez-alg-1} runs in $O(n \log n)$ time, simply
  observe that the first multiplication requires $O(n \log n)$
  operations per column of $\m{p}$; likewise, since
  $k = O(\sqrt{n \log n})$, the latter multiplication does, as well.
\end{proof}

An even simpler algorithm can be obtained by setting $k = 0$ in
\cref{alg:bez-alg-1}. This algorithm is referred to as the
Bernstein-Fourier algorithm elsewhere~\cite{bezerra2013efficient,
  delgado2015accurate}. In this case, the algorithm is just the update
$\m{p} \gets \m{C}_{n+1,n+1}^{(t)} \m{p}$. For both algorithms, we
note that (as is likely to be the case) the B\'{e}zier curve is to be
evaluated at multiple points, say
$0 \leq t_1 \leq \cdots \leq t_N \leq 1$, then the cost of the forward
FFT in applying $\m{C}_{n+1,n+1}^{(t)}$ is amortized with respect to
$N$, reducing the cost of batch evaluation of B\'{e}zier curves.


\section{Conclusion}

In this work, we presented novel algorithms for the multiplication and
inversion of the lower and upper Pascal matrices and their
$\ell_\infty$-normalized versions. In particular, we presented a full
complement of in-place $O(n^2)$ algorithms which are fast for small
problem sizes, and incorporated these into robust $O(n \log^2 n)$
algorithms which are competitive for all problem sizes. We numerically
demonstrated that previous algorithms based on $O(n \log n)$ Toeplitz
matrix factorization are too unstable to be useful. We also showed
that the Bernstein matrices generalize the $\ell_\infty$-normalized
pascal matrix, and that our algorithms correspond to De Casteljau's
algorithm for stable B\'{e}zier curve evaluation, thus providing a
novel $O(n \log n)$ algorithm for this task.

In the future, there are several avenues of work that we will explore:
\begin{itemize}
\item The Pascal matrix occurs repeatedly in the translation theory of
  the fast multipole method~\cite{fmm-book}, in particular for the
  one-dimensional Cauchy and Hilbert fast multipole methods. We will
  explore an implementation of the fast multipole method whose
  translation operators are expressed in terms of Pascal matrices.
\item The matrix representations of a variety of discrete polynomial
  transforms can be factored into products involving Pascal matrices
  and diagonal matrices~\cite{aburdene1994unification}. We will
  explore fast discrete polynomial transforms based on fast Pascal
  matrix algorithms and possible applications. In particular, a fast
  discrete Legendre transform effectively yields a fast spherical
  harmonic transform~\cite{healy2003ffts}.
\end{itemize}
\noindent It remains an open question whether a fast \emph{and stable}
$O(n \log n)$ Pascal matrix algorithm exists. Such an algorithm would
have manifold application, and this is also an avenue which we will
pursue in the future.

The code used to generate the results here as well as a library
containing a sample implementation of the algorithms used here will be
available publically presently.


\bibliographystyle{plain}
\bibliography{ms}{}

\end{document}